\documentclass[12pt,reqno]{amsart}
\usepackage[T1]{fontenc}
\usepackage{ae,aecompl}
\usepackage{graphics,cancel,graphicx}
\usepackage{epsfig,psfrag,manfnt}
\usepackage{mathrsfs}
\usepackage{graphicx,mathabx,color}
\usepackage{bbm,subfigure,rotating,bm,float,mathdots,wasysym,scalerel,xcolor}
\usepackage{url}

\usepackage[all,cmtip]{xy}
\usepackage{amssymb,amsmath,amsfonts,amsthm,color,xcolor}

\definecolor{cerulean}{rgb}{0.0, 0.48, 0.65}
\definecolor{burntorange}{rgb}{0.8, 0.33, 0.0}
\usepackage{scalerel,stackengine}
\stackMath
\newcommand\reallywidehat[1]{%
	\savestack{\tmpbox}{\stretchto{%
			\scaleto{%
				\scalerel*[\widthof{\ensuremath{#1}}]{\kern-.6pt\bigwedge\kern-.6pt}%
				{\rule[-\textheight/2]{1ex}{\textheight}}
			}{\textheight}%
		}{0.5ex}}%
	\stackon[1pt]{#1}{\tmpbox}%
}

\DeclareMathOperator{\Jac}{Jac}

\DeclareMathOperator{\codim}{codim}

\newcommand{\calA}{\mathcal{A}}

\newcommand{\calD}{\mathcal{D}}
\newcommand{\calE}{\mathcal{E}}

\newcommand{\calS}{\mathcal{S}}

\newcommand{\mC}{\mathbb{C}}

\newcommand{\mF}{\mathbb{F}}

\newcommand{\mN}{\mathbb{N}}

\newcommand{\mR}{\mathbb{R}}

\newcommand{\mZ}{\mathbb{Z}}

\newcommand{\supp}{{\textrm{supp}{\;}}}

\DeclareMathOperator{\U}{{\mathcal{U}}} 
\DeclareMathOperator{\Span}{span}

\newtheorem{theorem}{Theorem}[section]
\newtheorem{lemma}[theorem]{Lemma}
\newtheorem{corollary}[theorem]{Corollary}
\newtheorem{proposition}[theorem]{Proposition}

\newtheorem{remark}[theorem]{Remark}

\newtheorem{definition}[theorem]{Definition}

\theoremstyle{definition}

\theoremstyle{definition}

\begin{document}

\keywords{Gleason-Kahane-\.{Z}elazko Theorem, function algebras, generated by units, convolution rings of distributions}

\subjclass[2010]{Primary 46H05; Secondary 46F10, 15A86, 47B49}

 \title[]{On the Gleason-Kahane-\.{Z}elazko theorem for associative algebras}

\author[]{Moshe Roitman}
\address[]{Department of Mathematics\\
		University of Haifa 199 Abba Khoushy Avenue\\
		Mount Carmel, Haifa 3498838\\
		Israel}
\email{mroitman@math.haifa.ac.il}
 \author[]{Amol Sasane}
\address{Department of Mathematics \\London School of Economics\\
     Houghton Street\\ London WC2A 2AE\\ United Kingdom}
\email{A.J.Sasane@lse.ac.uk}

 \vspace{-0.81cm}
 \begin{abstract}
The classical Gleason-Kahane-\.{Z}elazko Theorem states that  a  linear functional on a complex Banach algebra   not vanishing on units, and such that $\Lambda(\mathbf 1)=1$, is multiplicative, that is, $\Lambda(ab)=\Lambda(a)\Lambda(b)$ for all $a,b\in A$. We study the GK\.Z  property for associative unital algebras, especially for function algebras.
In a  GK\.Z algebra $A$ over a field  of at least $3$ elements, and having an ideal of codimension $1$, every element  is a finite sum of units. A real or complex algebra with just countably many maximal left (right) ideals, is a GK\.Z  algebra. 
If $A$ is a commutative algebra, then the localisation $A_{P}$ is a GK\.Z-algebra for every prime ideal $P$ of $A$. Hence the GK\.Z property is not a local-global property.
The class of GK\.Z  algebras is closed under homomorphic images. If a function algebra $A\subseteq \mathbb F^{X}$ over a subfield $\mathbb F$ of $\mathbb C$, contains all the bounded functions in $\mathbb F^{X}$, then each element of $A$ is a sum of two units. If $A$ contains also a discrete function, then $A$ is a GK\.Z algebra.
We prove that the algebra of periodic distributions,  and the  unitisation of the algebra of distributions with support in $(0,\infty)$  satisfy the GK\.Z  property, while the algebra of compactly supported distributions does not. 
\end{abstract}
\maketitle

\section{Introduction}
We study the Gleason-Kahane-\.Zelazko property (Definition \ref{defGKZ} below) of associative unital algebras, especially of function algebras. We will also investigate the validity of the GK\.Z theorem in some  natural convolution algebras of distributions. 

\subsection{Background} 
The classical Gleason-Kahane-\.{Z}elazko Theorem,  which was proved independently by Gleason \cite{Gle}, and by Kahane and \.{Z}elazko \cite{KahZel}, provides a characterisation of the maximal ideals of a commutative complex Banach algebra $A$: a subspace of $A$ is a maximal ideal if and only if it has codimension $1$ and contains no units. Equivalently, a  linear functional $\Lambda: A\to \mathbb C$  not vanishing on units, and such that $\Lambda(\mathbf 1_{A})=1$, is multiplicative, that is, $\Lambda(ab)=\Lambda(a)\Lambda(b)$ for all $a,b\in A$. The formulation in  terms of  linear functionals was extended by \.Zelazko to any complex Banach algebra \cite{Zel68}, thus providing a characterisation of the ideals of codimension $1$ (that are necessarily maximal) in a complex Banach algebra.

\begin{definition}\label{defGKZ} 
	Let $A$ be a unital algebra over a field $\mathbb F$. 
	The algebra $A$ is said to have the {\bf GK\.Z property} (or to be a {\bf GK\.Z algebra}) if  every linear functional $\Lambda:A\rightarrow \mF$ not vanishing on units, and such that $\Lambda(\mathbf 1_{A})=1$ is multiplicative.
\end{definition}

In Definition \ref{defGKZ} we allow the case that $\mathbf 1_{A}=\mathbf 0_{A}$, that is the case that $A$ is the zero algebra because, although we always  start with a nonzero algebra, we may obtain sometimes zero algebras. By the above definition, a zero algebra is a GK\.Z  algebra.

The converse of the GK\.Z property is obvious: if $\Lambda$ is a multiplicative nonzero linear functional on $A$, then $\Lambda$ does not vanish on units and $\Lambda(\mathbf 1_{A})=1$.

\begin{remark}
	Let $A$ be an algebra over a field $\mathbb F$.	
	Every subspace of $A$ is contained in a subspace of codimension $1$. Thus the maximal subspaces of $A$ are the subspaces of codimension $1$.
	\end{remark} 

An algebra $A$ satisfies the GK\.Z-property  if and only if each subspace of $A$ of codimension $1$ not containing units is multiplicatively closed, equivalently it is an ideal (necessarily maximal).	

There is an extensive literature on extensions of the Gleason-Kahane-\.Zelazko Theorem. See, e.g., the surveys \cite{Jar91} and \cite{MasRan}.

\begin{definition} 
	A function algebra over a field $\mathbb F$ on a non-empty set $X$, is a unital subalgebra of the $\mathbb F$-algebra $\mathbb F^{X}$ (the set of all functions $X\to \mathbb F$, with addition and multiplication defined componentwise)	such that $\mathbf 1_{\mathbb F^{X}}\in A$.
\end{definition}  

\subsection{Summary}
In section \ref{sec.algebras} we deal with the GK\.Z  property for algebras, in section \ref{sec.funcalg} for function algebras, and in section \ref{sec.distributions} for distribution algebras.

The main result of \S\ref{subsec.genunits} ({\em The GK\.Z property of algebras and generation by units}), is  that  a  GK\.Z algebra $A$ over a field of at least $3$ elements and having an ideal of codimension $1$,  is generated by units, that is, every element of $A$ is a finite sum of units (Theorem \ref{gkzimpgeun} (1)),  equivalently, $A=\Span(\U(A))$. A GK\.Z  function algebra over a field $\mathbb F$ with at least $3$ elements is unit generated.

In  \S\ref{subsec.lincov} ({\em Linear coverings}), we prove that if the field $\mathbb F$ is infinite,  and $A\setminus \U(A)$ is  contained in a union of less than $|\mathbb F|$ proper one-sided ideals, then the algebra $A$ satisfies the GK\.Z  property. (Proposition \ref{linearav}). Thus a real or complex algebra with just countably many maximal left ideals is a GK\.Z  algebra. We also show that
if $A$ is a commutative algebra, then the localisation $A_{P}$ is a GK\.Z-algebra for every prime ideal $P$ of $A$. Hence the GK\.Z property is not a local-global property.

In \S\ref{subsec.GKZimg} ({\em The GK\.Z property for homomorphic images})  we show that  the class of GK\.Z  algebras is closed under homomorphic images (Proposition \ref{homimg}), and provide a sufficient condition for the GK\.Z  property of $A$ assuming that $A/I$ is a GK\.Z  algebra for an ideal $I$ of $A$ (Proposition \ref{A/I}). Naturally,  $A$ satisfies the GK\.Z  property if and only if  $A/\Jac(A)$ does (Proposition \ref{Jac}), where $\Jac(A)$  is the Jacobson radical of $A$.

In \S\ref{subsec.unitis} ({\em Unitisation}), we show that the unitisation of a radical algebra is a GK\.Z  algebra, and provide a  topological sufficient condition for an algebra to be radical (Proposition \ref{topUnitGKZ}).

In section \ref{sec.funcalg} ({\em On the  GK\.Z property of function algebras}), we prove that if a function algebra $A\subseteq \mathbb F^{X}$ over a subfield $\mathbb F$ of $\mathbb C$ contains all the bounded functions in $\mathbb F^{X}$, then each element of $A$ is a sum of two units. If $A$ contains also a discrete function, then $A$ is a GK\.Z algebra.

In section \ref{sec.distributions} ({\em Algebra of periodic distributions}), we prove that the algebra of  periodic distributions,  and the  unitisation of the algebra of distributions with support in $(0,\infty)$  satisfy the GK\.Z  property, while the algebra of compactly supported distributions does not. Actually, we started our paper with the study of the GK\.Z  property for some distribution algebras, and this study led us to more general results.

\subsection{Conventions and notations}
Unless otherwise stated, all the algebras in this paper are nonzero unital associative  algebras,
not necessarily commutative, over a field $\mathbb F$. We  use the notation $\mathbf 1_{A}$ for the unity of the algebra $A$, and $\U(A)$ for the group of units (invertible elements) of $A$. Using the $\mathbb F$-algebra monomorphism $\mathbb F\to A,\ f\mapsto f\mathbf 1_{A}$, we may assume that $\mathbb F\subseteq A$.
In the whole paper, we denote by $\mathbb F$ a field, by  $A$ a nonzero associative unital algebra over $\mathbb F$, and by $X$ a nonempty set. Thus a function algebra is a subalgebra of $\mathbb F^{X}$ with the same unity. We let $\mathbb F^{\bullet}=\mathbb F\setminus \{0\}$. The cardinality of a set $X$ is denoted by $|X|$. If $Y$ is a subset of $X$, we denote by $I_{Y}$ the indicator (characteristic function) of $Y$. If $A\subseteq \mathbb F^{X}$ is a function algebra, and $x\in X$, the projection $p_{x}$ is defined as the function $p_{x}: A\to X\ (f\mapsto f(x))$.
The field of two elements is denoted by $\mathbb F_{2}$.
A {\em proper} subspace of an algebra $A$ is a subspace of $A$ different from $A$, possibly zero. Similarly, we use {\em proper} ideals, etc. If $A$ is an algebra and $a\in A$, then $\rho (a)$ (the {\em resolvent} of $a$) is the set of all scalars $\lambda\in \mathbb F$ such that $a-\lambda\mathbf 1$ is invertible, and $\sigma(a)$ (the {\em spectrum}) of $a$ is the set $\mathbb F \setminus \rho(a)$. 

\section{The GK\.Z property of general  algebras}\label{sec.algebras}
\subsection{The GK\.Z property of algebras and generation by units}\label{subsec.genunits}
There are several equivalent ways to define the property of an algebra to be generated by units: as an algebra, as a ring, as a vector space (that  is, $A=\Span(\U(A))$, and as an additive group (which means that every element of $A$ is a finite sum of units). For two surveys on rings generated by units see \cite{Sri10} and \cite[Section 1]{Sri16}.

\begin{theorem}\label{gkzimpgeun} 
Let $A$ be a GK\.Z algebra over a field $\mF$ of at least $3$ elements. If $A$  has an ideal of codimension $1$,  (equivalently, if there exists a linear functional on $A$ not vanishing on units), then $A$ is generated by units.
\end{theorem}	

\begin{proof} 
			Let $\Lambda$  be a linear functional that does not vanish on units. Replacing $\Lambda$ by $\frac {\Lambda} {\Lambda(\mathbf 1)}$, we may assume that $\Lambda(\mathbf 1)=1$. Assume that $A$ is not generated by units, that is, $A\ne \Span(\U(A))$. There exists  a subspace $V$ of $A$  of codimension $1$ containing $\U(A)$.		 
			Let $a\in A\setminus V$.  Thus $A=\mF a\oplus V$. Let $a^{2}=\lambda a+w$, where $\lambda \in \mF$, and $w\in V$. Since $\mF$ contains at least  three elements, we may choose $c\in \mF$ such that 
			$$
			c^{2}- \lambda c-\Lambda(w)\ne0.
			$$ 
			There exists a unique linear functional $\Psi:A\rightarrow \mF$  such that  $\Psi(v)\!=\!\Lambda(v)$ for all $v\in V$, and $\Psi(a)= c$. 
			Hence 
			\[
			\Psi(a^{2})=\Psi(\lambda a+w)=\lambda c+\Lambda(w)\ne c^{2}=(\Psi(a))^{2},
			\]
			so $\Psi$ is not multiplicative.
			On the other hand, $\Psi(u)=\Lambda(u)\ne0$ for all $u\in \U(A)$, and $\Psi(\mathbf 1)=\Lambda(\mathbf 1)=1$, contradicting the assumption that $A$ has the GK\.Z property. 	
	\end{proof}

\begin{corollary}\label{funcUng} 
A GK\.Z  function algebra over a field $\mathbb F$ with at least $3$ elements is unit generated.	
\end{corollary}

\begin{proof} 
	Let $A\subseteq F^{X}$ be a nonzero function algebra. For all $x\in X$, the projection $p_{x}$ is a multiplicative functional, and $p_{x}(I_{X})=1$. By Theorem \ref {gkzimpgeun}, $A$ is  generated by units. 
\end{proof}  

An algebra $A$ is said to satisfy the GK\.Z property {\em vacuously} if each linear functional on $A$ vanishes on some unit, equivalently, if each maximal subspace of $A$ contains a unit. Clearly, if $A$ is a GK\.Z algebra vacuously, then $A$ is  a GK\.Z algebra.

\begin{remark}
	In Theorem \ref{gkzimpgeun}, the condition on the existence of an ideal of codimension $1$ is essential, but not necessary.
\end{remark}

Indeed, consider the $\mathbb C$-algebras $\mathbb C(x)$ and $\mathbb C(x)[y]$, where $x$ and $y$ are two independent indeterminates over $\mathbb C$. Both  these algebras  satisfy  the GK\.Z property vacuously by the next Lemma \ref{2fields}. However, $\Span(\U(\mathbb C(x))=\mathbb C(x)$, so the condition is not necessary, and $\Span(\U(\mathbb C(x)[y]))=\mathbb C(x)\ne \mathbb C(x)[y]$, so the condition is essential. \qed

\begin{lemma}\label{2fields} 
Let $\mathbb F \subsetneqq \mathbb K$ be two fields, and let $A$ be an algebra over $\mathbb K$.  Then  $A$ is a GK\.Z algebra vacuously over $\mathbb F$.	
\end{lemma}

\begin{proof} 
	Let $V$ be an $\mathbb F$-subspace of $A$ of codimension $1$. Since $\mathbb F \subsetneqq \mathbb K$, we have $[\mathbb K:\mathbb F]>1$, so $V$ contains a nonzero scalar $k\in\mathbb K$,  and $k\in \U(A)$.	
\end{proof}

Theorem \ref{gkzimpgeun} implies:
\begin{corollary}\label{gkzvac}  
	Let $A$ be an algebra over a field of cardinality $\ge3$, that is not generated by units. Then $A$ is a GK\.Z algebra $\Leftrightarrow A$ is a GK\.Z algebra vacuously. 
\end{corollary} 

\begin{remark}
If  $A$ is an algebra, then $\Span(\U(A))$ is a subalgebra of $A$ with the same units.
\end{remark}

\begin{proposition} \label{notgkz}
Let $\mathbb F$ be a field of at least $3$ elements.  
\begin{enumerate} 
	\item 
	The following three conditions are equivalent:
	\begin{enumerate} 
		\item 
		$A$ is a GK\.Z algebra.
		
		\item 
		$A$ is a GK\.Z algebra vacuously.
				
		\item
		$\Span(\U(A))$ is a GK\.Z algebra vacuously.
	\end{enumerate}
	
	\item 
	If $C\subsetneqq A$ are two algebras with the same units, then either both $A$ and $C$ are GK\.Z algebras vacuously, or both of them are not GK\.Z  algebras.
\end{enumerate}
\end{proposition}

\begin{proof}\
		
		\begin{enumerate} 
			\item			
			$(a)\Rightarrow (b)$ 
			
			If $A$ is not a GK\.Z algebra vacuously, then by Theorem \ref{gkzimpgeun}, $A=\Span(\U(A))$, a contradiction.
			
			$(b)\Rightarrow (a)$
			Clear.
			
			$(b)\Rightarrow (c)$
			
			If condition (c) does not hold, then there exists a linear functional $\Lambda$ on $C$  not vanishing on units and such that $\Lambda(\mathbf 1_{C})=1$. Thus $\Lambda$  can be extended to a linear functional $\widetilde\Lambda$ on $A$. We see that $\widetilde\Lambda$ does not vanish on units, since   $\U(A)=\U(C)$, and $\widetilde\Lambda(\mathbf 1_{A})=1$, contradicting $(b)$.
			
			$(c)\Rightarrow (b)$
			
			If (b) does not hold, then there exist a linear functional on $A$  not vanishing on units and preserving unity. Its restriction to $C$ has the same properties, contradicting (c).
						
			\item 
			Clearly, (1) implies (2).
		\end{enumerate}
\end{proof}  
 
 \begin{corollary} 
 If $A$ is a GK\.Z	 algebra, then all proper subalgebras of $A$ with the same units, are GK\.Z algebras vacuously.
 \end{corollary}  
 
 For group rings, see  \cite{MillSeghgrRin}.
 
 \begin{remark} 
 If $A$ is an algebra, then  $\Span(\U(A))$ is a homomorphic image of the group ring $\mathbb F[\U(A)]$. 
\end{remark}
 
\begin{corollary}\label{groring} 
No proper algebra extension of a group ring  $\mathbb F[G]$ with the same units  satisfies the GK\.Z property.
\end{corollary}  

\begin{proof} 
	This follows from Theorem \ref{gkzimpgeun}, since the augmentation homomorphism of $\mathbb F[G]$ is a multiplicative linear functional.
\end{proof} 

For an application of Corollary \ref{groring} to distribution algebras, see \S\ref{subsec.E'}.

 \begin{proposition}\
	 
	\begin{enumerate} 
		\item 
		$\mathbb F_{2}$ is the only field for which all algebras generated by units satisfy the GK\.Z  property.		
		
		\item 
		$\mathbb F_{2}$ is the only field for which the only unit generated function algebra is the field itself (up to isomorphism).
	\end{enumerate}
\end{proposition}

\begin{proof}\
	
\begin{enumerate} 
	\item 	
Let $A$ be a unit generated algebra over $\mathbb F_{2}$.	If $\Lambda: A\to \mathbb F_{2}$ is a linear functional not vanishing on units, then $\Lambda(u)=1$ for every unit $u$ in $A$. Since $A=\Span(\U(A))$, we obtain that $\Lambda$ is multiplicative, implying that $A$ is a GK\.Z  algebra. 
	
	On the other hand, if $|\mathbb F|>2$, let $c\in \mathbb F\setminus \{0,1\}$. Consider the Laurent polynomial ring $\mathbb F[x, x^{-1}]$.  Let $\Lambda: \mathbb F[x,x^{-1}]\to \mathbb F$ be the unique linear functional over $\mathbb F$, such that $\Lambda(x^{n})=1$ for all integers $n\ne2$, and $\Lambda(x^{2})=c$. We see that  $\Lambda$ is not multiplicative since $\Lambda(x^{2})=c\ne 1=(\Lambda(x))^2$. Since $\U(\mathbb F[x,x^{-1}])=\mF^{\bullet}\{x^{n}\mid n\in \mathbb Z\}$ (a direct product of two groups), we see that $\Lambda$ does not vanish on units. Hence the $\mathbb F$-algebra $\mathbb F[x,x^{-1}]$ is unit generated, but it is not a GK\.Z algebra.
	
	\item
Let $A\subseteq \mathbb F_{2}^{X}$ be a function algebra. The only unit of $\mathbb F_{2}^{X}$ is $I_{X}$. Hence $\Span(\U(A))=\mathbb F_{2}I_{X}$. It follows that $A=\Span(\U(A)) \Leftrightarrow A=\mathbb F_{2}I_{X}$.	

On the other hand, if $|\mathbb F|>2$,  then every GK\.Z  algebra over $\mathbb F$ is unit generated by Theorem \ref{gkzimpgeun}, and by (1), there are GK\.Z  algebras $A$ that are not isomorphic to $\mathbb F$. 
\end{enumerate}
\end{proof}  

As shown by Vamos \cite[page 418]{Vam05}, each element of a real or complex Banach algebra is a sum of two units, as an immediate consequence of the spectral  theorem. Analogously, we have:

\begin{remark}
Let $A$ be an $\mathbb F$-algebra.
Let $a\in A$ such that $\rho(a)\ne\emptyset$. Then $a$ is a sum of at most $2$ units.
Hence, if every element of $A$ has a non-empty resolvent, then each element of $A$ is a sum of at most $2$ units.	
\end{remark} 

Indeed, let $\lambda\in \rho(a)$. If $\lambda=0$, then $a$ is invertible. If $\lambda\ne0$, then $a=(a-\lambda \mathbf 1)+\lambda \mathbf 1$,  is a sum of two units. \qed

\subsection{Linear coverings}\label{subsec.lincov}
	\begin{remark}\label{onesmax}	
	A one-sided  ideal $V$ of codimension $1$ of the algebra $A$ is a maximal ideal.
	\end{remark}  

Indeed, $V$ is multiplicatively closed, so $V$ is an ideal.  \qed

\begin{lemma} \label{fgid}
Let $V$ be a subspace $A$ of codimension $1$ that is not an ideal. Then there exists a finitely generated subspace $V_{0}$ of $V$ such that $AV_{0}=V_{0}A=A$.	
\end{lemma}  

\begin{proof} 
By Remark \ref{onesmax}, $V$ is not a left ideal, so $V\subsetneqq AV$, implying that $AV=A$. Similarly, $VA=A$. Hence there exist finitely generated subspace $V_{1}$ and $V_{2}$ of $V$ such that $AV_{1}=V_{2}A=A$. Set $V_{0}=V_{1}\cup~V_{2}$.
\end{proof}

\begin{proposition}\label{linearav}
	Let $A$ be an $\mathbb F$-algebra. Then $A$ is a GK\.Z algebra under each of the following conditions:
	\begin{enumerate} 
		\item
		$A\setminus \U(A)$ is a contained in a union of at most two proper one-sided ideals.		
		
		\item 
		$\mathbb F$ is infinite,  and $A\setminus \U(A)$ is  contained in a union of less than $|\mathbb F|$ proper one-sided ideals.
		
		\item
		$\mathbb F$ is infinite,  and $A\setminus \U(A)$ is  contained in a union of finitely many proper one-sided ideals.
	\end{enumerate}
		\end{proposition}   

\begin{proof}	
		Let $V$ be a subspace of $A$ of codimension $1$ that contains no units.
		\begin{enumerate} 
		\item 
		We have $V\subseteq A\setminus \U(A)\subseteq I_{1}\cup I_{2}$, where $I_{1}, I_{2}$ are two one-sided ideals of $A$, not necessarily distinct. Viewing $A$ as an  additive group, we obtain that $V\subseteq I_{k}$ for some $k=1,2$. Since $\codim V=1$, we see that $V=I_{k}$, so $V$ is an ideal by Remark \ref{onesmax}. Hence $A$ is a GK\.Z  algebra.
				
		\item 
		 By Lemma \ref{fgid}, there exists a finitely generated subspace $V_{0}$ of $V$ such that $AV_{0}=V_{0}A=A$. Since $V_{0}$ is  contained in a union of less than $|\mathbb F|$ proper one-sided ideals, it follows  from\cite[Theorem 1.2]{Khare09} or from \cite[Main Theorem]{Clark12},	 that $V_{0}$ is contained in one of the one-sided ideals $I$ in this union. Hence $V=I$, so $V$ is an ideal by Remark \ref{onesmax}, and $A$ is a GK\.Z  algebra. 
		 
		\item is a particular case of (2).
		\end{enumerate}
\end{proof}	

\begin{corollary} 
	Let $A\subseteq \mathbb F^{X}$ be a function algebra. Assume that $|X|<|\mathbb F|$, and that each $f\in A$ such that $f(x)\ne 0$ for all $x\in X$ is invertible in $A$. Then  $A$, is a GK\.Z-algebra. In particular, if $|X|<|\mathbb F|$, then $\mathbb F^{X}$ is a GK\.Z  algebra. More particularly, if $\mathbb F=\mathbb R, \mathbb C$, and $X$ is countable, then $A$ is a GK\.Z-algebra. 
\end{corollary}  

\begin{proof} 
	We have $A\setminus \U(A)\subseteq \bigcup_{x\in X}p_{ x}^{-1}(0)$. Hence $A$ is a GK\.Z algebra by Proposition \ref{linearav}.	
\end{proof}  

\begin{proposition}\label{loc} 
Let $A$ be an algebra over a field $\mathbb F$.

\begin{enumerate} 
	\item	
	If $A$ has at most two maximal left (or at most two maximal right) ideals, then $A$ satisfies the GK\.Z property.
	
	\item
		Let $A$ be an algebra over an infinite field $\mathbb F$. If the cardinality of the set of left  ideals, or of the right ideals, is $<|\mathbb F|$, then $A$ is a GK\.Z algebra. 
		
	\item
	If  the field $\mathbb F=\mathbb R$ or $\mathbb C$, and $A$ has just countably many left ideals, then $A$ is a GK\.Z algebra.
	
	\item 
If $A$ is commutative, then the localisation $A_{P}$ is a GK\.Z-algebra for every prime ideal $P$ of $A$. Hence the GK\.Z property is not a local-global property.
	
	\item 
Assume that $A$ is commutative and that the field $\mathbb  F$ is  infinite. Let $\mathcal P$ be a set of less that $|\mathbb F|$ prime ideals. Then
the algebra $B=\bigcap_{P\in\mathcal P}A_{P}$ satisfies the GK\.Z property.	
\end{enumerate}
\end{proposition}

\begin{proof}
	\
	 
	Items (1), (2) and (3) immediately follow from Proposition \ref{linearav}.
	
	\begin{enumerate} \setcounter{enumi}3
		
	\item $A_{P}$ has just one maximal ideal, so it a GK\.Z-algebra by (1).
	
	\item
	This follows from Proposition \ref{linearav} since the set of non-units of the algebra  $B$ is contained in $\bigcup_{P\in\mathcal P}PA_{P}$.	
\end{enumerate}
\end{proof}

\subsection{The GK\.Z property for homomorphic images}\label{subsec.GKZimg}
We show first that the class of GK\.Z algebras is closed under homomorphic images (Proposition \ref{homimg}). Equivalently, if $I$ is an ideal of a GK\.Z algebra $A$, then also $A/I$ is a GK\.Z algebra. Then we deal with the converse, and show that under suitable assumptions, if $A/I$ is a GK\.Z algebra for a certain ideal $I$, then $A$ is a GK\.Z algebra.

\begin{proposition} \label{homimg}
	A homomorphic image of a GK\.Z algebra is a GK\.Z algebra.	
\end{proposition}  

\begin{proof}	
	We have to show that if $I$ is an ideal of a GK\.Z  algebra $A$, then $A/I$ is a GK\.Z algebra.	Let $W$ be a subspace of $A/I$ of codimension $1$ that contains no units. Hence $W=V/I$, where $V$ is a subspace of $A$ of codimension $1$, that contains $I$, but no units modulo $I$. We infer that  $V$ contains no units of $A$, implying that $V$ is an ideal of $A$, since $A$ is a GK\.Z  algebra. Hence  $W=V/I$ is an ideal of $A/I$. It follows that $A/I$ is a GK\.Z algebra.
\end{proof}

 For a partial converse of Proposition \ref{homimg}, see the next Proposition \ref {A/I}. Recall that if $I$ is an ideal of $A$, an element $t\in A/I$ is {\em liftable} to a unit in $A$ if $t$ contains a unit in $A$, that is, $t+I=u+I$ for some $u\in \U(A)$.

\begin{proposition} \label{A/I}
	Let $A$ be an algebra over $\mathbb F$.
	\begin{enumerate} 
		\item 
		Let $I$ be an ideal of  $A$. Then $A/I$ is a GK\.Z algebra  $\Leftrightarrow$  each maximal subspace of $A$ containing $I$, but no units modulo $I$, is an ideal. 			
		
		\item 
		If each unit of $A/I$ is liftable to $A$, then  $A$ is a GK\.Z algebra  $\Leftrightarrow A/I$  is a GK\.Z algebra. 	
	\end{enumerate}	
\end{proposition}

\begin{proof}\
\begin{enumerate} 
	\item 
			$\Rightarrow$ 
	
	See Proposition \ref{homimg}.		
	
			$\Leftarrow$
		
		Let $V$ be  a maximal subspace of $A$ containing $I$, but no units modulo $I$. Thus $V/I$ is a maximal subspace of $A/I$ that contains no units. Hence $V/I$ is an ideal of $A/I$, implying that $V$ is an ideal of $A$.	
	\item follows from (1) since in this case a subspace of $A$ containing $I$, contains no units in $A$ if and only if it contains no units modulo $I$.	
\end{enumerate}		
	\end{proof} 

 By Proposition \ref{loc} (1), if $A$ is a unital algebra with a unique maximal left ideal $J$, then $A$ satisfies the GK\.Z property. For another generalisation, see Proposition \ref{Jac} (2) below. Indeed,  in this case  $J=\Jac(A)$, the Jacobson radical of $A$. Recall that $\Jac(A)$ is the intersection of all maximal left ideals of $A$, and that in this definition `left ' can be changed to `right', so $\Jac(A)$ is a two-sided ideal. For the Jacobson radical theory, see e.g., \cite[section 13B]{IsaacsAlg}, \cite[Part II, \S 1]{KaplFiRings},  and \cite[Chapter 2.4]{Lam01}. 
 
 \begin{proposition}
  Each maximal subspace $V$ of $A$ without units contains $\Jac(A)$, so $\Jac(A)$ is the intersection of all maximal subspaces of $A$ without  units.	
\end{proposition}  	

\begin{proof} 
	There exists a unique linear functional $\Lambda$ on $A$ such that $\ker \Lambda=V$ and $\Lambda(\mathbf 1)=1$. Let $c\in \Jac(A)$. We have $\Lambda(c)\in \sigma(c)=\{0\}$. Hence $\Lambda(c)=0$, so $c\in V$.
\end{proof}  
 
 	\begin{proposition}\label{Jac} 
 		If $I$ is an ideal contained in $\Jac(A)$, then $A$ is a GK\.Z algebra  $\Leftrightarrow A/I$ is a GK\.Z algebra.
  	\end{proposition}

\begin{proof}	
	This follows from Proposition \ref{A/I}, since $I\subseteq \Jac(A)$, so an element of $A$ is a unit in $A$ if and only if it is a unit modulo $I$.
\end{proof}    

\begin{proposition}\label{avoide} 
	Let $I$ be an ideal of the algebra $A$. Assume that each maximal subspace of $A$ that does not contain units and does not contain $I$, is an ideal in $A$. Equivalently, assume that if $\Lambda$ is a linear functional such that $\Lambda(u)\ne0$ for all $u\in\U(A)$, $\Lambda(\mathbf 1_{A})=1$ and $\Lambda(I)\ne(0)$, then $\Lambda$ is  multiplicative. Then
	$A$ is a GK\.Z algebra $\Leftrightarrow A/I$ is a GK\.Z algebra.
\end{proposition}

 \begin{proof} 
 	In view of Proposition \ref{homimg}, we have to prove just the implication $\Leftarrow$. Let $V$ be a  maximal subspace of $A$ that does not contain units and does not contain $I$. If $V$ is not an ideal, then $I\subseteq V$ and $V/I$ is not an ideal in $A/I$, a contradiction. Thus every maximal subspace of $A$ is an ideal, so $A$ is a GK\.Z algebra.
 \end{proof}   

\begin{remark}\label{emptyspec}
If  $A$ contains an element $a$ with empty spectrum, then $A$ is GK\.Z algebra vacuously. 	
\end{remark} 

Indeed, if $\Lambda$ is a linear functional that does vanish on units such that $\Lambda(\mathbf 1)=1$, then $\Lambda(a)\in \sigma(a)=\emptyset$, a contradiction. \qed  

\begin{proposition}\label{gkzlift}
Let $I$ be an ideal of a generated  by units algebra $A$ satisfying the following two conditions:	
\begin{enumerate} 
	\item
	There exists an element $c\in A$ such that $c+\lambda{\mathbf 1}$ is invertible modulo $I$ for all $\lambda\in \mathbb F$.

	\item 
	Each unit in $A/I$ is liftable to a unit in $A$.	
	\end{enumerate}

Then $A$ is a GK\.Z algebra.
\end{proposition} 

\begin{proof} 
	By (1), the element $c+I$ has an empty spectrum in $A/I$. Hence, $A/I$ is a GK\.Z algebra by Remark \ref{emptyspec}. By Proposition  \ref{A/I}, $A$ is a GK\.Z algebra.
\end{proof}  
 
\subsection{Unitisation}\label{subsec.unitis}
Recall that the {\em unitisation} of an $\mathbb F$-algebra $A$ (not necessarily unital) is the algebra $B=\mathbb F\oplus A$ (a direct sum of additive groups) with  multiplication defined by $(\lambda+a)(\mu +b)=\lambda\mu+\lambda b+\mu a+ ab$ for $\lambda,\mu\in \mathbb F$, and $a,b\in A$. We identify $\mathbb F$ with the subfield $\mathbb F\mathbf 1_{B}$ of $B$.

For the next proposition recall that if $I$ is an ideal of an algebra $A$, not necessarily unital,  then $J\subseteq \Jac(A)$  if and only if each element $a\in I$ is left  quasiregular, that is,  $1+a$ is left invertible in the unitisation of $A$, equivalently, there exists an element $b\in A$ such that $a+b+ab=0$ (see e.g., \cite{QuasiWiki21}). Moreover,  an ideal $I$ of $A$ is contained in $\Jac(A)$ if and only if each element of $I$ is left invertible. Here, `left' can be replaced by `right' or omitted (\cite{KaplFiRings} and \cite{IsaacsAlg}). Hence, a non zero non-unital algebra  $A$ is {\em radical} (that is, $A=\Jac(A)$), if and only if each element of $A$ is quasiregular.
 
\begin{lemma}\label{unitisrad}
	Let $A$ be a nonzero $\mathbb F$-algebra, not necessarily unital, and let $B$ the unitisation of $A$.  The following conditions are equivalent:
	
	\begin{enumerate} 
		\item 
		$A$ is radical algebra.		
	
		\item 
		$A=\Jac(B)$.
			
		\item
		$1+A\subseteq \U(B)$ (here $1\in \mathbb F\subseteq B$, so $1\notin A$).
		
		\item
		$\U(B)=\mathbb F^{\bullet}(1+A)$ (a direct product of multiplicative groups).
		
		\item
		$\U(B)=B\setminus A$.
	\end{enumerate}	
\end{lemma}

\begin{proof} 
We have: \\
 $A=\Jac(A) \Leftrightarrow $ each element of $A$ is quasiregular $ \Leftrightarrow  A=\Jac(B) \Leftrightarrow (3)$. 
 
 Thus the first three conditions are equivalent. Since  $B=\mF+A$, we see that the last three conditions are also equivalent.\end{proof}
	
\begin{proposition}\label{topUnitGKZ} 
The  unitisation of  a radical algebra is a GK\.Z algebra. 
Moreover,  $A$ is a radical algebra if there exists a Hausdorff topology on $A$  such that $\lim_{n\to\infty} a^{n}=0$ in this topology for all $a\in A$. 	
\end{proposition}  

	\begin{proof} 
	Assume that $A$ is a radical algebra, and $B$ is the unitisation of $A$. By Lemma \ref{unitisrad},  $A=\Jac(B)$. Since $B/\Jac(B)=B/A\cong \mathbb F$, it follows that $B$ satisfies the GK\.Z property by Proposition \ref {Jac}. Alternatively, it follows from Lemma \ref{unitisrad} that $\U(B)=B\setminus A$, so $A$ is the unique maximal one-sided ideal of $A$. By Proposition \ref{linearav}, $B$ is a GK\.Z algebra.
		
Assume that $A$ has a Hausdorff topology as described.	Let $a\in A$. Since $\lim_{n\to\infty}a^{n}=0$, the geometric series $\sum_{n=0}^{\infty}a^n$ converges in $A$, and its sum $s$  satisfies in $B$: $s(1-a)=(1-a)s=1$, so $1-a$ is invertible in $B$. By Lemma \ref{unitisrad}, $B$ is a GK\.Z-algebra.
	\end{proof}

For an application of Proposition \ref{topUnitGKZ} to distribution algebras see \S\ref{subsec.unitis}.

\section{On the  GK\.Z property of function algebras}\label{sec.funcalg}
In this section we investigate the GK\.Z  property for function algebras, using ideas related to  connections between ideals of a commutative ring and ultrafilters, although we do not use ultrafilters explicitly.
 Instead of ideals, we use maximal subspaces with no units. For ideals and ultrafilters  see the classical book \cite{GilJer},  for later results see \cite{MaaBenhFilt},  and  for the basics of ultrafilter theory see \cite{UltrafWiki}. 

Let $A$ be an algebra over a subfield $\mathbb F$ of $\mathbb C$. We prove that if $\mathbb F$ contains all the bounded functions in $\mathbb F^{X}$, then each element of $A$ is a sum of two units (Proposition \ref{subf2units}). If $A$ contains also a discrete function, then $A$ is a  GK\.Z algebra (Corollary \ref{bounDisGkz}). The condition that $A$ contains all bounded functions is implied by the following condition: if $f\in A$, $g\in \mathbb F^{X}$, and $|g|\le |f|$ (that is, $|g(x)|\le |f(x)|$ for all $x\in X$), then $g\in A$.

\goodbreak 
For the next Lemma \ref{allzero}, see \cite{RecovUltra} and especially Eric Wofsey's answer. In this lemma, the implications $(1) \Rightarrow (2) \Rightarrow (3)$  are easily proved.  Nevertheless, we indicate all the relevant implications.

\begin{lemma}\label{allzero}
Let $A$ be a commutative algebra, and let $\Lambda$ be a linear functional on  $A$ such that $\Lambda(u)\ne0$ for all $u\in\U(A)$ and $\Lambda(\mathbf 1)=1$. Let $e$ be an idempotent in $A$.  We have:
	\begin{enumerate} 
		\item
		Let $u,v\in \U(A)$. Then 
		\begin{enumerate} 
			\item 
			$ue+v(1-e)\in \U(A)$.
			
			\item 
			Exactly one of the two scalars $\Lambda(ue)$ and $\Lambda(v(1-e))$ vanishes.
		\end{enumerate}
				
		\item 
		\begin{enumerate} 
			\item 
			$\Lambda(e)\in \{0,1\}$.
			
			\item 
		For all $u\in\U(A)$,   $\Lambda(ue)=0$ if $\Lambda(e)=0$, and $\Lambda(u)=\Lambda(ue)\ne0$ if $\Lambda(e)=1$.
		
		\item
		$\Lambda(ue)=\Lambda(u)\Lambda(e)$ for all $u\in \U(A)$ and idempotents $e\in A$.
		\end{enumerate}
						
		\item 		
		If $A$ is generated by units, then 	
		\begin{enumerate} 
			\item 
			$\Lambda(ae)=0$  for all $a\in A$ and idempotents $e\ in \ker\Lambda$, so			
		$\ker \Lambda$ contains the ideal generated by the idempotents it contains. 
		
			\item 
			$\Lambda(ae)=\Lambda(a)\Lambda(e)$ for all $a\in A$ and idempotents $e\in A$.					
		\end{enumerate}	
	\end{enumerate}
	\end{lemma}

\begin{proof}$\;$
	\begin{enumerate}
		\item 
		\begin{enumerate} 
			\item 
			$(ue+v(1-e))(u^{-1}e+v^{-1}(1-e))=1$, so $ue+v(1-e)\in \U(A)$.
			
			\item 
			 Since $\Lambda(e+(1-e))=1$, the scalars $\Lambda(e)$ and $\Lambda(1-e)$ cannot both vanish.
			 
			  Suppose that both of these scalars are nonzero. Then
			$
			\Lambda\left(\Lambda(v(1-e))u(e)-\Lambda(ue)v(1-e)\right)=0
			$, although  $\Lambda(v(1-e))u(e)-\Lambda(ue)v(1-e)\in \U(A)$ by (1) (a), a contradiction.
		\end{enumerate}		
		
		\item
		\begin{enumerate} 
			\item follows from (1)(b) applied to $u=v=1$.
			
			\item 
			 follows from (1)(b) applied to $u$ and to $v=1$. 			
			
			\item follows from (2)(b).
		\end{enumerate}
	
		\item
		\begin{enumerate} 
			\item follows from (2)(b).		
		
			\item follows from (2)(c).		
		\end{enumerate}	
		\end{enumerate}
\end{proof} 

\begin{proposition}\label{2cases}
	Let $A$ be a function algebra that is generated by units. 
	Suppose that $A$ contains the indicator functions of all subsets of $X$. Let $\Lambda$ be a linear functional such that $\Lambda(u)\ne 0$ for all $u\in\U(A)$ and $\Lambda(\mathbf 1)=1$.  If   $\Lambda(I_{\{x_{0}\}})\ne0$ for some $x_{0}\in X$, then $\Lambda=p_{x_{0}}$.  Hence $\Lambda(I_{x})=0$ for all $x\ne x_{0}$ in $X$.
\end{proposition}  

\begin{proof} 
Suppose that $\Lambda(I_{\{x_{0}\}})\ne0$ for some $x_{0}\in X$.	Since $\Lambda(I_{\{x_{0}\}})=1$ by Lemma \ref{allzero}(2)(a), we obtain  by Lemma \ref{allzero} (3) (b) for every $f \in A$: $\Lambda(f)\Lambda(I_{\{x_{0}\}})=\Lambda(f\;\!I_{\{x_{0}\}})=\Lambda(f(x_{0})I_{\{x_{0}\}})=f(x_{0})\Lambda(I_{\{x_{0}\}})=f(x_{0})=p_{x_{0}}(f)$.		
\end{proof}

\begin{remark}\label{cases} 
	If a function algebra $A$ on a set $X$ contains the indicators of all subsets of $X$, then we may use definition by cases to obtain a function in $A$.
\end{remark} 
More precisely,  given  a partition of $X$ into $n$ disjoint subsets $S_{1}, \dots, S_{n}$, and functions $f_{1}, \cdots, f_{n}$ in $A$, then the function $f=\sum_{i=1}^n f_{i}I_{S_{i}}$ belongs to $A$. \qed 

\begin{proposition}\label{subf2units} 
	Let $A$ be a function algebra over a subfield $\mathbb F$ of $\mathbb C$. Assume that $A$ contains all the bounded functions in $\mathbb F^{X}$. Then every element of $A$ is a sum of two units.	
\end{proposition}  

\begin{proof} 
	Let $f\in A$. We define two functions $g$ and $h$ in $\mathbb F^{X}$ as follows:
	\[
	g(x)=
	\begin{cases}
		f(x)+3	&\text{if } |f(x)|\le 2\\
		f(x)-1			&\text{ if }	|f(x)|>2
	\end{cases}
	\]
	
	\[
	h(x)=
	\begin{cases}
		\frac	1{f(x)+3} 	&\text{if } |f(x)|\le 2\\
		\frac 1	{f(x)-1}			&\text{ if }	|f(x)|>2.
	\end{cases}
	\]		
	We have $g\in A$ by Remark \ref{cases}, and $h\in A$ since $h$ is a bounded function. Also $gh=I_{X}$, so $g$ is invertible. We have $f=(f-g)+g$, and $f-g$ is invertible since $\frac {1} {f-g}$ is bounded. Hence $f$ is a sum of two units in $A$.	
\end{proof}  

\begin{corollary}\label{absBound} 
	Let $A$ be a function algebra over a subfield $\mathbb F$ of $\mathbb C$. Assume that  if $f\in A, g\in \mathbb F^X$, and $|g|\le |f|$, then $g\in A$. Then $A$ contains all the bounded functions in $\mathbb F^{X}$, so every element of $A$ is a sum of two units.
\end{corollary} 

\begin{proof} 
	Let $g$ be a bounded function in $\mathbb F^{X}$. Let $n$ be a positive integer such that $|g(x)|\le n$ for all $x\in X$. Thus $|g|\le nI_{X}$, implying that $g\in A$. By Proposition \ref{subf2units}, every element of $A$ is a sum of two units in $A$.	
\end{proof}

\begin{remark}\label{zerosetcof} 
In the setting of the next Proposition \ref{modL}, the ideal $L$ is the set generated by the indicators of all finite subsets of $X$. Thus $L$ consists of all functions  $t\in A$ such that $t(x)=0$ for all $x\in X$, except finitely many elements, that is, the set $t^{-1}(0)$ is cofinite in $X$.		
\end{remark} 

\begin{proposition}\label{modL}
Let $A\subseteq \mathbb F^{X}$ be a function algebra that contains the indicators of all singletons of $X$, and let $L$ be the ideal of $A$ generated by these indicators.  Then 	
\begin{enumerate} 
	\item 
	$A$ is a GK\.Z algebra  $\Leftrightarrow A/L$ is a GK\.Z algebra (Here  $A/L$ is  the zero algebra when $L=A$).
	
	\item
	If there exists in $A$ an element $f$ such that $f-\lambda I_{X}$ is invertible modulo $L$ for every scalar $\lambda\in \mathbb F$, then the only linear functionals $\Lambda$ on $A$ such that $\Lambda(u)\ne0$ for all $u\in\U(A)$ and $\Lambda(\mathbf 1)=1$,  are the projections $p_{x}$ for $x\in X$. Thus $A$ is a GK\.Z algebra.
\end{enumerate}
\end{proposition}  

\begin{proof} \ 
	 
	 \begin{enumerate} 
	 	\item 
	  By Proposition \ref{2cases}, the only linear functionals $\Lambda$ on $A$ such that $\Lambda(u)\ne0$ for all $u\in\U(A)$,  $\Lambda(\mathbf 1)=1$, and $\Lambda(L)\ne(0)$ are the projections $p_{x}$ for $x\in X$. Hence, $A$ is a GK\.Z algebra  $\Leftrightarrow A/L$ is a GK\.Z algebra.
	  
	 	\item
	By assumption, the spectrum of $f+L$ in $A/L$ is empty, so $A/L$ is a GK\.Z algebra by Remark \ref{emptyspec}. By (1), $A$ is a GK\.Z algebra.
\end{enumerate}	
\end{proof}  

\begin{remark} 
Here is an alternative proof of Proposition \ref{modL} using Proposition \ref{gkzlift}.	
\end{remark} 

By Proposition \ref{modL}, it is enough to show that if $f\in A$ and $f$ is a unit modulo $L$, then there exists $f_{0}\in\U(A)$ such that $f-f_{0}\in L$. By assumption, there exists $g\in A$ such that $fg=1+t$ for some $t\in L$. Define for $x\in X$:
\[f_{0}(x)=
\begin{cases}
	f(x) &\text{ if } t(x)=0 \text{ (that is, if } f(x)g(x)=1)\\
	1 & \text{ if } \text{otherwise}.
\end{cases}					
\]
Define $g_{0}$ similarly. Clearly, $f_{0}g_{0}=I_{X}$, and if $t(x)=0$ for some $x\in X$, then $(f-f_{0})(x)=0$, so by Remark \ref{zerosetcof}, $f-f_{0}\in L$, as required.	\qed

\begin{corollary}
	If the set $X$	is finite, then a subspace of $\mF^{X}$ that  contains  the indicators of all singletons of $X$ is equal to $\mF^{X}$. Hence, by Proposition \ref{modL}, $\mathbb F^{n}$  is a GK\.Z  algebra for all  positive integers $n$.
\end{corollary}	

Let $\mathbb F$ be a subfield of $\mathbb C$. A function $f\in \mathbb F^{X}$ is  called {\em discrete} if the set $f(X)$ is discrete in $\mathbb C$, and the sets $f^{-1}(\lambda)$ are finite for all $\lambda\in \mathbb C$.

\begin{remark} 
If $\mathbb F^{X}$ contains a discrete function, then $|X|=\aleph_{0}$. On the other hand, if $X=\{x_{n}\ (n\in \mathbb N)\}$, where the elements $x_{n}$ are distinct, and $f\in \mathbb F^{X}$, then $f$ is discrete if and only if $\lim_{n\to \infty}|f(x_{n})|=\infty$.
\end{remark} 

\begin{theorem} \label{main}
	Let $A\subseteq \mathbb F^{X}$ be a unit generated function algebra over a field $\mathbb F$ contained in $\mathbb C$. Assume that $A$ contains the indicators of all subsets of $X$, and a discrete function $f$. Then the only linear functionals $\Lambda$ on $A$ such that $\Lambda(u)\ne0$ for all $u\in\U(A)$ and $\Lambda(\mathbf 1)=1$,  are the projections $p_{x}$ for $x\in X$. Thus $A$ is a GK\.Z algebra. 
\end{theorem}

\begin{proof}
	 By Proposition \ref{modL} and in the same notation, it is enough to show that $f-\lambda \mathbf 1$ is invertible modulo $L$ for all $\lambda\in \mathbb F$. Define the function $f_{\lambda}:X\to \mathbb C$ as follows:
	\[	f_{\lambda}(x)=	\begin{cases}
						1		&\text{if } f(x)=\lambda I_{X}\\ 
		 				f(x)-\lambda I_{X}	&\text{if } f(x)\ne\lambda I_{X}.
	\end{cases}
	\]
	Since $f$ is discrete, $f(x)=\lambda I_{X}$ just for finitely many $x$'s. Hence  
	\[
	(f(x)-\lambda I_{X})-f_{\lambda}\in L.
	\]  Since the function $f_{\lambda}$ is discrete, there exists $m=|\min f_{\lambda}(X)|$, and $m>0$, implying that   $|\frac {1} {f_{\lambda}}(x)|\le \frac {1} {m}$, so $\frac {1} {f_{\lambda}}\in A$, and $f_{\lambda}$ is invertible in $A$. Since ${f(x)-\lambda I_{X}}\equiv f_{\lambda} \pmod L$, we infer that ${f(x)-\lambda I_{X}}$ is invertible modulo $L$ as required.
\end{proof}

\begin{corollary}\label{bounDisGkz}
	Let $A\subseteq \mathbb F^{X}$ be a function algebra over a subfield $\mathbb F$ of $\mathbb C$. Assume that  $A$ contains all the bounded functions in $\mathbb F^{X}$, and a discrete function $f$. Then every element of $A$ is a sum of two units,  the only linear functionals $\Lambda$ on $A$ such that $\Lambda(u)\ne0$ for all $u\in\U(A)$ and $\Lambda(\mathbf 1)=1$,  are  projections. Thus $A$ is a GK\.Z algebra.
\end{corollary}  

\begin{proof} 
	$A$ contains the indicators of all subsets of $X$ since the indicators are bounded functions. By Proposition \ref{subf2units}, every element of $A$ is a sum of two units. We conclude the proof by Theorem \ref{main}.
\end{proof}

\begin{corollary}\label{absGkz} 
	Let $A\subseteq \mathbb F^{X}$ be a function algebra over a subfield $\mathbb F$ of $\mathbb C$. Assume that  if $f\in A, g\in \mathbb F^X$, and $|g|\le |f|$, then $f\in A$, and that $A$ contains a discrete function. Then every element of $A$ is a sum of two units, $A$ contains all the bounded functions in $\mathbb F^{X}$, and the only linear functionals $\Lambda$ on $A$ such that $\Lambda(u)\ne0$ for all $u\in\U(A)$ and $\Lambda(\mathbf 1)=1$,  are the projections $p_{x}$ for $x\in X$. Thus $A$ is a GK\.Z algebra.
\end{corollary}    

 \begin{proof}  
 	By Corollary \ref{absBound}, $A$ contains all bounded functions. We conclude the proof by Corollary \ref{bounDisGkz}.
 \end{proof}  

\begin{corollary}\label{gkzfunalg} 
	Let $\mathcal H$ be a set of functions in $[2,\infty)^X$ satisfying the following conditions:
	
	\begin{enumerate}
		\item
		For each $h_{1}, h_{2}\in \mathcal H$ there exists $h\in \mathcal H$ such that for all $x\in X$, 
		$h_{1}(x)+h_{2}(x)\le h(x)$.
		\item
		For each $h_{1}, h_{2}\in \mathcal H$ there exists $t\in \mathcal H$ such that  for all $x\in X$
		$h_{1}(x)h_{2}(x)\le t(x)$.
	\end{enumerate}
	Let $A=\{f\in \mF^{X}\mid \exists h\in \mathcal H\textrm{ such that  }\forall x\in X, \;|f(x)|\le h(x)\}$.
	We have:
	\begin{enumerate} 
		\item[(a)]
		
		$A$ is a unital $\mF$-subalgebra of $\mathbb F^{X}$ containing $\mathcal H$.
		
		\item[(b)]
		
		Every element of $A$ is a sum of two units.
		
		\item[(c)]
		If $A$ contains a discrete function, then the only linear functionals $\Lambda$ on $A$ such that $\Lambda(u)\ne0$ for all $u\in\U(A)$ and $\Lambda(\mathbf 1)=1$,  are projections. Thus $A$ is a GK\.Z algebra.
	\end{enumerate} 	
\end{corollary} 

\begin{proof} 
	$\;$
	
	\noindent (a)
	By	item (1), we obtain inductively that for every $f\in A$ and every integer $n\ge1$, we have $|nf|\le h$ for some $h\in \mathcal H$, so $\lambda f\in A$ for all $\lambda\in \mF$. It is now easy to complete the proof of (a).	
	We conclude the proof using Corollary \ref{absBound}.	
\end{proof}

For an application of Corollary \ref{gkzfunalg} to distribution algebras see \S\ref{subsec.period}.

\section{The GK\.Z property for algebras of distributions}\label{sec.distributions}
 As usual, $\calD(\mR)$ is the space of test functions (compact supported complex-valued infinitely differentiable functions on $\mR$), and $\calD'(\mR)$ is the space of distributions on $\mR$.  We study three subspaces of $\mathcal D^{\prime}(\mathbb R)$ that are algebras with convolution as multiplication.

\subsection{Algebra of periodic distributions}\label{subsec.period}
For background on periodic distributions and its Fourier
series theory, we refer the reader to \cite[p. 527-529]{Tre}.
For $v\in {{\mathbb{R}}}$, the {\em translation operator} $\mathbf{S}_v:\calD'(\mR)\rightarrow \calD'(\mR)$,  is given by 
$$
\langle
{\mathbf{S}_{v}}t,\varphi\rangle = \langle
t,\varphi(\cdot+{v})\rangle\textrm{ for all }\varphi\in
{\mathcal{D}}({\mathbb{R}}), \textrm{ and for all } t \in \calD'(\mR).
$$ 
A distribution $t\in {\mathcal{D}}'({\mathbb{R}})$ is called {\em
	periodic with period}
$v\in {\mathbb{R}}\setminus \{\mathbf{0}\}$ if
$
{\mathbf{S}_v}t=t.
$ 
We define ${\mathcal{D}}'_{v}(\mR)$ to be the set of the periodic distributions with period $v$. As  is well known, ${\mathcal{D}}'_{v}(\mR)$ is a complex algebra with convolution as multiplication. Moreover,  using  Fourier transforms, one obtains that ${\mathcal{D}}'_{v}(\mR)$ is isomorphic as a complex algebra to   the algebra $\calS'(\mZ)$  of all complex-valued maps on $\mZ$ of at most polynomial growth, that is,
$$
\calS'(\mZ):=\bigg\{\mathbf{a}: \mZ\rightarrow \mC\;\Big|\;
\begin{array}{ll} \exists \textrm{ an integer }k\geq 1 \;\textrm{ such that} \\
	\forall n\in \mZ, \;
	|\mathbf{a}(n)|\leq  (2+|n|)^k
\end{array}\bigg\}.
$$
Since $\calS'(\mZ)$ is a function algebra satisfying the assumptions of  Theorem \ref{gkzfunalg}, by letting $\mathcal H$
to be the set of all functions $h: \mathbb Z\to [2, \infty)$ of the form $h(n)=(2+|n|)^{k}$ for all integers $n$, where $k\ge1$ is an integer, we obtain
\begin{proposition}
	All the linear functionals on the complex function algebra $\calS'(\mZ)$ not vanishing on units and preserving unity, are projections. Hence  the algebra $\calD'_v(\mR)$ of periodic distributions with period $v\in \mathbb R\setminus \{0\}$ satisfies the GK\.Z property. 	
\end{proposition}  

\subsection{The algebra $\mC\delta_0+\calD^{\prime}_{+}(\mR)$} 
\label{subsec.unitis}

\noindent Let $\calD^{\prime}_{+}(\mR)$ denote the set of all distributions $t \in \calD^{\prime}(\mR)$ 
having their distributional support $\supp t$ contained in the half line $(0,\infty)$. Then $\calD'_+(\mR)$ is an algebra without unity, with convolution as multiplication.  Let $\delta_0$ denote the Dirac distribution supported at $0\in \mR$.  Clearly,  $\calA := \mC \delta_0 + \calD'_+(\mR)$ is a unital  subalgebra of $D^{\prime}(\mathbb R)$, and $\mathcal A$ is isomorphic to the unitisation of $\calD'_+(\mR)$. 
\begin{proposition} 
	The unique linear functional on $\mathcal A$ not vanishing on units and preserving unity,  is the functional induced by the canonical homomorphism $\mathcal A\to \mathcal A/\calD'_+(\mR) =\mathbb C$. Hence the algebra $\mC\delta_0+\calD'_+(\mR)$ satisfies the GK\.Z property.
\end{proposition} 

\begin{proof} 
	Since $\mathcal A$ is a subalgebra of the topological algebra $\calD'_+(\mR)$, by Proposition \ref{topUnitGKZ}, it is enough to prove that $\lim_{n\to\infty}a^{\ast n}=0$ for all $a\in \mathcal A$.	This follows from  $\supp (s^{\ast n})\subseteq n\;\!\supp s$ for all positive integers $n$,  and so, for any test function $\varphi \in \calD(\mR)$,
	$(\supp (s^{\ast n}))\cap (\supp \varphi)\neq \emptyset$ for only finitely many $n\in \mN$. 
\end{proof}

\subsection{The algebra $\calE'(\mR)$} 
\label{subsec.E'}

\noindent Let $\calE'(\mR)$ denote the space of all distributions $t \in \calD'(\mR)$ 
that have compact support.  Thus $\calE'(\mR)$ is an algebra with convolution as multiplication.

The Dirac distribution with support equal to $\{a\}$, where $a\in \mR$, will be denoted by $\delta_a$.  We let $G=\{\delta_{a} \mid a\in \mathbb R\}$. Thus $G$ is a group isomorphic  to $(\mathbb R,+)$ by the map
$\mathbb R\to G\ (a\mapsto \delta_{a})$.

For the sake of completeness, in the next Proposition \ref{inv}, we reproduce the characterisation  of units in $\mathcal E^{\prime}(\mR)$ and its proof   from \cite{Sas0}:

\begin{proposition}\cite[Proposition 4.2]{Sas}\label{inv} 
	\label{prop_PW}
	$\U(\calE'(\mR))=\mathbb C^{\bullet}G$ (a direct product of two groups). \end{proposition}
\begin{proof}
	Clearly, $\mathbb C^{\bullet}G\subseteq \U(\calE'(\mR))$.
	For the converse inclusion, suppose that $t$ is invertible in $\calE'(\mR)$. Then there exists a distribution $s\in \calE'(\mR)$ such that 
	$t\ast s=\delta_0$. By the theorem on supports \cite[Theorem~4.3.3, p.107]{Hor}, we have 
	$$
	\textrm{c.h.supp}(t\ast s)=\textrm{c.h.supp}(t)+\textrm{c.h.supp}(s),
	$$ 
	where, for a distribution  $a\in \calE'(\mR)$, the notation $\textrm{c.h.supp}(a)$ is used for the closed convex hull of 
	$\supp a$, that is, the intersection of all closed convex sets containing $\supp a $. 
	So we obtain 
	$$
	\{0\}=\textrm{c.h.supp}(\delta_0)= \textrm{c.h.supp}(t\ast s)=\textrm{c.h.supp}(t)+\textrm{c.h.supp}(s),
	$$
	from which it follows that $\textrm{c.h.supp}(t)=\{a\}$ and  
	$\textrm{c.h.supp}(s)=\{-a\}$ for some $a\in \mR$. But then 
	also  $\supp t=\{a\}$ and  
	$\supp s=\{-a\}$. As distributions with 
	support in a point $p$ are linear combinations of ${\scaleobj{0.9}{\delta_p}}$ 
	and its derivatives ${\scaleobj{0.9}{\delta_p^{{\scaleobj{0.81}{(n)}}}}}$ \cite[Theorem~24.6, p.266]{Tre}, 
	$t$ and $s$ have the form 
	$$
	t={\scaleobj{0.93}{\sum_{n=0}^N}}\;\! t_n \delta_a^{(n)}, \textrm{ and } 
	s={\scaleobj{0.93}{\sum_{m=0}^M}}\;\! s_m \delta_{-a}^{(m)}, 
	$$
	for some integers $N,M\geq 0$ and $t_n,s_m\in \mC$ ($0\leq n\leq N$, $0\leq m\leq M$). 
	Then $t\ast s=\delta_0$ implies that $N=M=0$ and $t_0s_0=1$, since for each $a\in \mathbb R$, the elements $\delta_{c}^{(n)}$, where $c\in \{0,-a, a\}$ and $n\ge1$ is an integer, are linearly independent over $\mathbb C$ (if $a=0$, then $c=0$).
	 It follows that  $t_0\neq 0$. 
	Thus we have that 
	$
	t= t_0 \delta_a\in \{c\delta_p: p\in \mR, \;0\neq c\in \mC  \}.
	$  
	This completes the proof.
\end{proof}

\begin{proposition} 
	  $\calE'(\mR)$ is  not  a GK\.Z  algebra.
\end{proposition}  

\begin{proof} 
	Using Proposition \ref{inv}, we see	that the distributions in $\Span(\U(\calE'(\mR)))$ have  finite support, so $\Span(\U(\calE'(\mR)))$ is properly  contained in  $\calE'(\mR)$. Since the distributions $\delta_{a}\  (a\in\mathbb R)$ are linearly independent over $\mathbb C$, it follows that  the algebra $\Span(\U(\calE'(\mR)))$ is isomorphic to the group ring $\mathbb C[G]$. From Corollary \ref{groring} it follows that $\calE'(\mR)$ is not a GK\.Z algebra.
\end{proof} 

We can prove directly that $\calE'(\mR)$ is not a GK\.Z  algebra. 

Let   $\varphi$ be a function in $C^{\infty}(\mathbb R)$,  and let $\Lambda=\Lambda_{\varphi}: \calE'(\mR)\to \mathbb C$ be  the functional    $\Lambda(t)=\langle t, \varphi\rangle$ for all $t \in \calE'(\mR)$.  We have for all $a\in \mathbb R$: $\Lambda(\delta_a)=\varphi(a)$.  For all $a,b\in \mathbb R$: $\Lambda(\delta_{a})\cdot \Lambda(\delta_{b})=\Lambda(\delta_{a+b})=a+b$, and $\Lambda(\delta_{a})\Lambda(\delta(b))=\varphi(ab)$. Hence we have:

\begin{enumerate} 
	\item 
$\Lambda$ preserves units if and only if  $\varphi(a)\ne0$ for all $a\in \mathbb R$.
	
	\item 
$\Lambda(\delta_0)=1$	if and only if $\varphi(0)=1$.

	\item
	$\Lambda(uv)=\Lambda(u)(v)$ for all $u,v\in \U(\calE'(\mR))$  if and only if $\varphi(a+b)=\varphi(a)\varphi(b)$ for all $a,b\in \mathbb R$.
\end{enumerate}	
	The only continuous functions $\varphi$ with range contained in $\mathbb R$ satisfying the conditions stated in (2) and (3) above are the functions $c^{x}$, where $c\in (0,\infty)$. Thus the function 
	  $\varphi(x)=1+x^{2}$ for $x\in \mathbb R$ does not satisfy this conditions, but it does satisfy the the properties in (1) and (2). Hence  the corresponding functional $\Lambda_{\varphi}$ preserves
	   units and the unity, but it is not  multiplicative.
It follows that  $\mathcal E^{\prime}(\mathbb R)$ is not a  GK\.Z  algebra. 

\noindent 
For every $\varphi$ the linear functional on $\Lambda:\calE'(\mR)\to \mathbb C$ given by  $\Lambda(t)=\langle t, \varphi\rangle$ for all $t \in \calE'(\mR)$, is continuous, when $\calE'(\mR)$ is equipped with the weak dual/weak-$\ast$ topology $\sigma(\calE'(\mR),\calE(\mR))$; see e.g. \cite{Tre}.

\end{document}